\documentclass{article}
\usepackage{graphicx} 

\usepackage{graphicx}%
\usepackage{multirow}%
\usepackage{amsmath,amssymb,amsfonts}%
\usepackage{amsthm}%
\usepackage{mathrsfs}%
\usepackage[title]{appendix}%
\usepackage{xcolor}%
\usepackage{textcomp}%
\usepackage{manyfoot}%
\usepackage{booktabs}%
\usepackage{algorithm}%
\usepackage{algorithmicx}%
\usepackage{algpseudocode}%
\usepackage{listings}%

\theoremstyle{thmstyleone}%
\newtheorem{theorem}{Theorem}
\newtheorem{corollary}[theorem]{Corollary}%
\newtheorem{lemma}[theorem]{Lemma}
\newtheorem{conjecture}[theorem]{Conjecture}%

\theoremstyle{thmstyletwo}%
\newtheorem{remark}{Remark}%

\theoremstyle{thmstylethree}%
\newtheorem{definition}{Definition}%

\newcommand{\setR}{\mathbb{R}}

\let\altphi\phi
\let\phi\varphi
\let\varphi\altphi
\let\altphi\undefined

\newcommand{\di}{\mathop{}\!\mathrm{d}}

\DeclareMathOperator{\RCD}{RCD}

\newfont{\tmpf}{cmsy10 scaled 2500}

\newcommand{\intav}{{\mathop{\int\kern-10pt\rotatebox{0}{\textbf{--}}}}}

\renewcommand{\ }{\text{ }}
\def\<{\langle}
\def\>{\rangle}

\newcommand{\R}{\setR}

\title{Poincar\'e inequality for one forms on four manifolds with bounded Ricci curvature}
\author{Shouhei Honda
\thanks{Graduate School of Mathematical Sciences, The University of Tokyo, shouhei@ms.u-tokyo.ac.jp
}
and Andrea Mondino
\thanks{Mathematical Institute, University of Oxford, andrea.mondino@maths.ox.ac.uk
}}

\begin{document}

\maketitle

\begin{abstract}
    In this short note, we provide a quantitative global Poincar\'e inequality for one forms on a closed Riemannian four manifold, in terms of an upper bound on the diameter, a positive lower bound on the volume, and a two-sided bound on Ricci curvature. This seems to be the first non-trivial result giving such an inequality without any higher curvature assumptions. The proof is based on a Hodge theoretic result on orbifolds, a comparison for fundamental groups, and a spectral convergence with respect to Gromov-Hausdorff convergence, via a degeneration result to orbifolds by Anderson.
\end{abstract}

\section{Introduction}
Let $M$ be a closed Riemannian $n$-manifold, 
let $\Delta^{H, k}$ be the Hodge Laplacian acting on $k$-forms on $M$ and denote by
\begin{equation*}
    0\le \lambda^{H, k}_0 \le \lambda^{H, k}_1 \le \lambda^{H, k}_2 \le \cdots \to \infty
\end{equation*}
the $i$-th eigenvalues $\lambda^{H, k}_i=\lambda_i^{H,k}(M)$ of $\Delta^{H, k}$ counted with their multiplicities. Let $\nu^{H, k}$ be the first positive eigenvalue, namely
\begin{equation*}
    \nu^{H, k}=\min \{\lambda^{H, k}_i | \lambda^{H, k}_i >0\}.
\end{equation*}
It is natural to ask whether there exists a quantitative positive lower bound on $\nu^{H, k}$ in terms of a certain restriction on curvature and sizes of $M$, for example its diameter $\mathrm{diam}=\mathrm{diam}(M)$ and its volume $\mathrm{vol}=\mathrm{vol}(M)$. 

In the case when $k=0$, namely, considering the (minus) Laplacian $-\Delta=-\mathrm{tr}(\mathrm{Hess})$ acting on functions, it is well-known 
that
\begin{equation}\label{function}
    \nu^{H, 0}=\lambda_1^{H, 0} \ge C(n, D, K)>0
\end{equation}
where $D$ is an upper bound on $\mathrm{diam}(M)$ and $K$ is a lower bound on Ricci curvature $\mathrm{Ric}$.

In the case when $k \ge 1$, Gallot-Meyer \cite{GM} proved a similar estimate under the assumption of a positive lower bound of the curvature operator. Moreover, Colbois-Courtois \cite{CC}  proved a lower bound on $\nu^{H,k}$ for every $k\geq 1$ assuming a 2-sided bound on the sectional curvature, an upper bound on the diameter, and a lower bound on the volume. 

Next let us focus on the case $k=1$. Recalling the Bochner formula for one forms:
\begin{equation}\label{bochneroneform}
\frac{1}{2}\Delta |\omega|^2 = |\nabla \omega|^2- \langle \Delta^{H, 1}\omega, \omega \rangle +\langle \mathrm{Ric}, \omega \otimes \omega \rangle, 
\end{equation}
it is natural to restrict Ricci curvature in order to get a quantitative lower estimate on $\nu^{H, 1}$. However, in authors' knowledge, there is no such a result, except for a trivial case, $\mathrm{Ric} \ge K>0$, which yields $\nu^{H, 1}\ge K$ just by integrating (\ref{bochneroneform}) over $M$ for an eigen-one-form $\omega$.

The main purpose of this note is to provide a positive result along this direction in dimension $4$, more precisely, we will show that 
\begin{equation*}
    \nu^{H, 1} \ge C(D, v)>0
\end{equation*}
under the assumptions: $n=4$, $\mathrm{diam} \le D<\infty$, $\mathrm{vol} \ge v>0$ and $|\mathrm{Ric}|\le 1$. See Theorem \ref{thm:positiveeigenvalue}.


\section{Main Result and Proof}

Let us start by recalling the definition and some basic properties of orbifolds, we refer the reader to \cite{Satake, Kleiner-Lott, Caramello} for a more comprehensive introduction.

A \emph{local model} is a pair $(\hat{U}, G)$,  where $\hat{U} \subset \R^n$ is a connected open subset and $G$ is a finite group  acting smoothly and effectively on $U$, on the right (\emph{effectiveness}
means that the homomorphism $G \to {\rm Diff}(\hat{U}_\alpha)$ is injective).

A \emph{smooth map} between local models  $(\hat{U}_1, G_1)$ and  $(\hat{U}_2, G_2)$ is the datum of  a smooth map $\hat{f}: \hat{U}_1\to \hat{U}_2$ and a homomorphism $\varphi:G_1\to G_2$ such that $\hat{f}$ is $\varphi$-equivariant, i.e. $\hat{f}(x g_1)=\hat{f}(x) \varphi(g_1)$. Let us explicitly point out that $\varphi$ is not assumed injective or surjective. An \emph{embedding} is a smooth map between local models so that $\hat{f}$ is an embedding; in this case, the effectiveness implies that $\varphi$ is injective.

\begin{definition}[Orbifold]
An \emph{atlas} for an $n$-dimensional orbifold $\mathcal{O}$ is the datum of:
\begin{itemize}
\item A Hausdorff paracompact space $|\mathcal O|$;
\item An open covering $\{U_\alpha\}$ for $|\mathcal O|$;
\item For each $U_\alpha$ there exists a local model $(\hat{U}_\alpha, G_\alpha)$ and a homeomorphism $\phi_\alpha: U_{\alpha}\to \hat{U}_\alpha/ G_\alpha$ with the following property: if $x\in U_1\cap U_2$, then there is a local model $(\hat{U}_3, G_3)$ with $x\in U_3$, where $U_3\subset U_1\cap U_2$ is homeomorphic to $\hat{U}_3/G_3$, together with local embeddings $(\hat{U}_3, G_3)\to (\hat{U}_1, G_1)$ and  $(\hat{U}_3, G_3)\to (\hat{U}_2, G_2)$.
 \end{itemize}
 Two atlases are equivalent if
they are both included in a third atlas. 
 
 An \emph{orbifold} $\mathcal{O}$ is an equivalence class of atlases.
 \end{definition}
 
 An orbifold $\mathcal{O}$ (with a given atlas) is \emph{oriented} if each $\hat{U}_\alpha$ is oriented, moreover the action of $G_\alpha$ as well as the embeddings $\hat{U}_3\to \hat{U}_1$ and $\hat{U}_3\to \hat{U}_2$   preserve the orientation.  The orbifold $\mathcal{O}$  is compact (resp. connected) if $|\mathcal O|$ is so.
 
 For a given  point $x\in |\mathcal O|$ and local model $(\hat{U}, G)$ around $x$, consider $\hat{x}\in \hat{U}$ which projects to $x$.  The \emph{local
group} $G_x$ is the stabilizer group $\{g\in G\colon \hat{x} g=\hat{x}\}$. It is always possible to find a local model so that $G=G_x$.  
 
 The regular part $|\mathcal{O}|_{reg}\subset |\mathcal O|$ is the subset of points $x\in  |\mathcal O|$ such that $G_x=\{1_G\}$. $|\mathcal{O}|_{reg}$ is a smooth manifold, moreover it is  open and  dense in  $|\mathcal O|$.
 
 A \emph{smooth map} (resp. $C^{k,\alpha}$-map) $f: \mathcal{O}_1\to \mathcal{O}_2$ between orbifolds is given by a continuous map $|f|:|\mathcal{O}_1|\to |\mathcal{O}_2|$  such that for each $x\in \mathcal{O}_1$ there are local models $(\hat{U}_1, G_1)$ and $(\hat{U}_2, G_2)$ around $x$ and $f(x)$, and a smooth  (resp. $C^{k,\alpha}$-map) $\hat{f}$ between the local models so that the diagram with the projections onto $U_1$ and $U_2$ commutes.
 
 \begin{remark}\label{c2alpha}
 One can define the \emph{tangent bundle} $T\mathcal{O}$ as a suitable orbifold, which coincides with the usual tangent bundle when restricted to the regular part (we refer to \cite{Kleiner-Lott} for the precise definitions). Similarly for the cotangent bundle $T^*\mathcal{O}$, and their tensor products. A section for such bundles  is said to be \emph{smooth} (resp. of class $C^{k,\alpha}$) if it is so as a map between orbifolds.
 In particular, we will use such a terminology in the proof of the main theorem, when saying that a Riemannian metric $g$ on $\mathcal{O}$ is of class $C^{1,\alpha}$.
\end{remark}

On an orbifold $\mathcal{O}$, one can define differential forms and study de Rham cohomology \cite{Satake} (see also \cite{Caramello} for an introduction). We denote by $H^k_{dR}(\mathcal{O})$ the $k^{th}$-de Rham cohomology group of $\mathcal{O}$.

The following lemma is  well-known to experts, we include it with a proof for the reader's convenience.

\begin{lemma}\label{coincidence}
Let $\mathcal{O}$ be an $n$-dimensional orientable orbifold. Then
 the  first Betti number $b_1(|\mathcal O|)$, defined as the rank over $\mathbb{Z}$ of the  fundamental group $\pi_1(|\mathcal O|)$, coincides with the rank over $\R$ of the first de Rham cohomology group $H^1_{dR}(\mathcal{O})$.
\end{lemma}

\begin{proof}
 By Poincar\'e duality for the de Rham cohomology of orbifolds \cite[Theorem 3]{Satake}, it holds that $H^1_{dR}(\mathcal{O})=H^{n-1}_{dR}(\mathcal{O})$.
 
The de Rham  Theorem for orbifolds \cite[Theorem 1]{Satake}, gives that $H^{n-1}_{dR}(\mathcal{O})=H^{n-1}(|\mathcal O|, \R)$, where the latter is the $(n-1)$-th real singular cohomology group. 

 By Poincar\'e duality with coefficients in $\R$ \cite{Satake}, it holds that the $(n-1)$-th real singular cohomology group and the $1^{st}$ real singular homology group coincide: $H^{n-1}(|\mathcal O|, \R)=H_1(|\mathcal O|, \R)$.

 Finally, It is a general fact that  ${\rm Rk}_{\R}(H_1(X,\R))={\rm Rk}_{\mathbb{Z}}(\pi_1(X))$ for any topological space $X$, where the latter is the rank over $\mathbb{Z}$ of the (classical) fundamental group.

\end{proof}

We are now in a position to prove the main result of the note. 
\begin{theorem}\label{thm:positiveeigenvalue}
Let $M$ be a closed Riemannian manifold of dimension at most $4$. If $|\mathrm{Ric}| \le 1$, $\mathrm{diam} \le D<\infty$ and $\mathrm{vol} \ge v>0$, then 
\begin{equation*}
\nu^{H, 1} \ge C(D, v)>0.
\end{equation*} 
Thus, in other words, the global Poincar\'e inequality:
\begin{equation}\label{Poincare}
\inf_{\eta}\int_M|\omega-\eta|^2\di \mathrm{vol} \le C(D, v)^{-1}\int_M\left( |\delta \omega|^2 +|d\omega|^2\right)\di \mathrm{vol}
\end{equation}
holds for any smooth $1$-form $\omega$ on $M$, where the infimum in the left-hand-side of (\ref{Poincare}) runs over all harmonic one forms $\eta$ on $M$.
\end{theorem}
\begin{proof}
Note that it is enough to consider the case when the dimension is $4$ because other cases can be reduced to the 4D case after taking the product of a lower dimensional flat torus (or just apply a result of \cite{CC}).

Firstly let us discuss the case when $M$ is orientable.
We will argue by contradiction.  If it is not the case, then there exist a sequence of closed oriented Riemannian $4$-manifolds $M_i$ with $|\mathrm{Ric}| \le 1$, $\mathrm{diam} \le D<\infty$ and $\mathrm{vol}\ge v>0$ such that 
\begin{equation}\label{positiveeigenvalue}
\nu^{H, 1}(M_i) \to 0.
\end{equation}

It follows from \cite{Anderson, Anderson2, CheegerColding1} with \cite{CN} that after passing to a subsequence, there exists a compact Riemannian $4$-orbifold $\mathcal{O}$ endowed with a Riemannian metric $g \in W^{2,p}$ for any $p<\infty$ (thus, in particular, $g\in C^{1, \alpha}$ for any $\alpha<1$), such that $M_i$ measured Gromov-Hausdorff (mGH) converge to $\mathcal{O}$ and the limit measure is the Hausdorff measure of dimension $4$.  
Note that  $\mathcal{O}$ is orientable by \cite{HondaOrientation}.
From \cite{HondaSpectral}, the spectral convergence for the Hodge Laplacian acting on $1$-forms holds for the convergence $M_i \to \mathcal{O}$, namely $\lambda^{H, 1}_j(M_i) \to \lambda^{H, 1}_j( \mathcal{O})$ for any $j$.  
In particular
\begin{equation*}
\limsup_{i \to \infty}b_1(M_i)\le b_1^H(\mathcal{O}),
\end{equation*}
where $b^H_1(\mathcal O)$ denotes the dimension of the space of harmonic one forms on $\mathcal O$.
Moreover the inequality must be strict:
\begin{equation*}
\limsup_{i \to \infty}b_1(M_i)< b_1^H(\mathcal{O}),
\end{equation*}
because of (\ref{positiveeigenvalue}).
On the other hand,  applying \cite[Theorem 1.1]{SormaniWei} for any sufficiently large $i$, there exists a surjective homomorphism from $\pi_1(M_i)$ to $\pi_1(|\mathcal O|)$; here we used that the revised fundamental group (used in \cite{SormaniWei})  of an orbifold coincides with the classical fundamental group, since every orbifold is semi-locally simply connected (as each point admits a contractible neighbourhood). In particular, considering their ranks, we have
\begin{equation*}
\liminf_{i \to \infty}b_1(M_i) \ge b_1(|\mathcal O|).
\end{equation*}
Since, by Hodge theory, linearly independent harmonic 1-forms on $\mathcal O$ generate different classes in $H^1_{dR}(\mathcal O)$, 
Lemma \ref{coincidence} yields
\begin{equation}\label{harmonicbetti}
    b_1^H(\mathcal O) \le b_1(|\mathcal O|).
\end{equation}
As a summary
\begin{equation*}
\limsup_{i \to \infty}b_1(M_i) < b_1^H(\mathcal O)\le b_1(|\mathcal O|)\le \liminf_{i \to \infty}b_1(M_i),
\end{equation*}
which is a contradiction.

Finally, let us discuss the case when $M$ is non-orientable.
Take the two-sheeted orientable Riemannian covering $\pi:\tilde M \to M$. Note that any harmonic one form on $\tilde M$ comes from the pull-back of a harmonic one form on $M$ because of $b_1(\tilde M)=b_1(M)$. Thus applying (\ref{Poincare}) to $\tilde M$ completes the proof.
\end{proof}
As an immediate consequence of the proof above we obtain the following.
\begin{corollary}\label{b1convergence}
    Let $M_i$ be a sequence of Riemannian $4$-manifolds with $|\mathrm{Ric}| \le 1$, Gromov-Hausdorff converging to a $4$-dimensional $C^{1, \alpha}$-Riemannian orbifold $\mathcal{O}$. Then for any sufficiently large $i$:
    \begin{equation*}
        b_1(M_i)=b_1^H(\mathcal{O})=b_1(|\mathcal{O}|)=b_1(|\mathcal{O}|_{reg}).
    \end{equation*}
\end{corollary}

Let us make some small observation on other related spectral gap results. Let $\mathbb{R}^n/\mathbb{Z}^n$ be  the $n$-torus with the canonical flat Riemannian metric $g$. Then one can easily find a smooth family of Riemannian metrics $g_{\epsilon}$ on $\mathbb{R}^n/\mathbb{Z}^n$ satisfying that $g_{\epsilon}$ has no parallel vector fields and that $g_{\epsilon}$ converge smoothly to $g$ as $\epsilon \to 0$. This observation tells that in general, for any dimension,  under any negative lower bound and any positive upper bound on any curvature, there is no quantitative spectral gap for the connection Laplacian $\nabla^*\nabla$ acting on one forms.

On the other hand, in the case when  the Ricci curvature is non-negative, any harmonic one form must be parallel because of (\ref{bochneroneform}). Combining this observation with the trivial point-wise inequality:
\begin{equation*}
    |d\omega| + |\delta \omega| \le C |\nabla \omega|
\end{equation*}
for any smooth one form $\omega$, we obtain the following.
\begin{corollary}\label{spectralgapconnection}
    Let $M$ be a closed Riemannian manifold of dimension at most $4$ with $0 \le \mathrm{Ric} \le 1$, $\mathrm{diam} \le D<\infty$ and $\mathrm{vol} \ge v>0$. Then the first positive eigenvalue, denoted by $\nu^{C, 1}$, of $\nabla^*\nabla$ acting on one forms satisfies
    \begin{equation*}
        \nu^{C, 1} \ge C(D, v)>0.
    \end{equation*}
\end{corollary}
Note that thanks to the spectral convergence results established in \cite{HondaSpectral}, we can also give the corresponding results for limit spaces.

Finally let us conclude the note by proposing a (maybe well-known for experts) conjecture. Notice that, in sharp contrast with the case of functions \eqref{function},   it is not possible to remove the lower bound on the volume, see \cite{CC}. 
\begin{conjecture}
    Let $M$ be a closed Riemannian $n$-manifold with $\mathrm{Ric} \ge -1$, $\mathrm{diam} \le D<\infty$ and $\mathrm{vol}\ge v>0$. Then
    $$
    \nu^{H, 1}\ge C(n, D, v)>0.
    $$
\end{conjecture}
It is also natural to ask the same question for non-collapsed  $\RCD$ spaces. 

\smallskip\noindent
\textbf{Acknowledgement.}
This work was conducted at the The Erwin Schrödinger International Institute for Mathematics and Physics (ESI), where both the authors were partially supported to participate in the workshop, Synthetic Curvature Bounds for Non-Smooth Spaces: Beyond Finite Dimension. We would like to express our gratitude to the organizers of the workshop and to the ESI for providing an excellent working environment.

The first named author acknowledges support of the Grant-in-Aid for Scientific Research (B) of 20H01799, the
Grant-in-Aid for Scientific Research (B) of 21H00977 and Grant-in-Aid for Transformative
Research Areas (A) of 22H05105.

The second named author acknowledges support by the European Research Council (ERC), under the European Union Horizon 2020 research and innovation programme, via the ERC Starting Grant  “CURVATURE”, grant agreement No. 802689.

\bibliography{sn-bibliography}

\end{document}